\documentclass[12pt]{amsart}

\usepackage{amssymb,amsmath,amsthm}

\newtheorem{theorem}{Theorem}[section]

\newtheorem{prop}[theorem]{Proposition}
\newtheorem{cor}[theorem]{Corollary}
\newtheorem{defn}{Definition}[section]

\def \R{\mathbb{R}}

\def \N{\mathbb{N}}

\def \Om{\Omega}
\def \Lap{\triangle}
\def \grad{\nabla}

\numberwithin{equation}{section}

\begin{document}

\title{A Note on the Transport Method for Hybrid Inverse Problems}

\author[Chung]{Francis J. Chung}
\address{Department of Mathematics, University of Kentucky, Lexington, KY, USA}
\email{fj.chung@uky.edu}

\author[Hoskins]{Jeremy G. Hoskins}
\address{Department of Mathematics, Yale University, New Haven, CT, USA}
\email{jeremy.hoskins@yale.edu}

\author[Schotland]{John C. Schotland}
\address{Department of Mathematics and Department of Physics, University of Michigan, Ann Arbor, MI, USA}
\email{schotland@umich.edu}

\begin{abstract}
There are several hybrid inverse problems for equations of the form $\grad \cdot D \grad u - \sigma u = 0$ in which we want to obtain the coefficients $D$ and $\sigma$ on a domain $\Omega$ when the solutions $u$ are known. One approach is to use two solutions $u_1$ and $u_2$ to obtain a transport equation for the coefficient $D$, and then solve this equation inward from the boundary along the integral curves of a vector field $X$ defined by $u_1$ and $u_2$.  It follows from an argument given by Bal and Ren in \cite{BalRen2} that for any nontrivial choices of $u_1$ and $u_2$, this method suffices to recover the coefficients on a dense set in $\Om$. This short note presents an alternate proof of the same result from a dynamical systems point of view.  
\end{abstract}

\maketitle

\section{Introduction}

Suppose $\Omega$ is a smooth bounded domain in $\R^n$, and $f \in C^{\infty}(\partial \Omega)$. Let $D$ be a uniformly positive function on $\Omega$, and $\sigma$ be a nonnegative function on $\Omega$, and consider the problem
\begin{equation}\label{CauchyProblem}
\begin{split}
\grad \cdot D \grad u - \sigma u   &=  0 \mbox{ on } \Omega\\
                   u|_{\partial \Omega} &= f. \\
\end{split}
\end{equation}
For these notes we will consider $D \in C^1(\Omega)$ and $\sigma \in C(\Omega)$.

In several hybrid inverse problems involving equations of this type, we can take advantage of physical phenomena to recover the solution $u$ to \eqref{CauchyProblem} for a given boundary condition $f$, without a priori knowledge of $D$ and $\sigma$ \cite{BalRen, BalRen2, BalSch, ChuHosSch, MclZhaMan}. To complete these problems, we need a method of recovering $D$ and $\sigma$ from the solutions $u$.

One approach \cite{BalRen2, BalUhl2, RenGaoZha} is to note that the equation in \eqref{CauchyProblem} can be written out as 
\begin{equation}\label{Transport}
D \Lap u + \grad D \cdot \grad u - \sigma u = 0.
\end{equation}
If $u$ is known, this can be viewed as a transport equation for $D$ with coefficients determined by $u$. Indeed, if we have two known solutions $u_1$ and $u_2$ to \eqref{Transport}, we can multiply the equation for $u_1$ by $u_2$ and vice versa, and subtract the two to obtain 
\begin{equation}\label{PureTransport}
D \left(u_2\Lap u_1 - u_1\Lap u_2\right) + \grad D \cdot \left(u_2\grad u_1 - u_1\grad u_2\right) = 0.
\end{equation}
This eliminates $\sigma$ to provide a transport equation for $D$ with known coefficients. Assuming we can measure $D|_{\partial \Om}$, then it follows from the basic theory of transport equations (\cite{Eva}, Ch. 3) that we can solve \eqref{PureTransport} to obtain $D$ on all of the integral curves of the vector field 
\begin{equation}\label{XDef}
X := u_2\grad u_1 - u_1\grad u_2
\end{equation}
that intersect the boundary of $\Om$.  Once $D(x)$ is known, we can solve for $\sigma(x)$ using \eqref{Transport}. Note that the maximum principle implies that if $u$ is positive on the boundary, then $u$ must be positive inside the domain, eliminating the possibility of difficulties if $u(x) = 0$.

The major potential problem with this transport method is the possibility that not every point in $\Omega$ can be reached from the boundary by following an integral curve of $X$.  In \cite{BalUhl2}, the authors use the existence of complex geometrical optics (CGO) solutions to \eqref{Transport} to show that there exist boundary conditions $f_1$ and $f_2$ for which the corresponding solutions $u_1$ and $u_2$ give rise to a vector field $X$ whose boundary-intersecting integral curves cover $\Omega$. However, the rapid exponential decay of CGOs can be difficult to work with in practice.  

Fortunately, it turns out that any non-trivial positive boundary conditions yield a pair of solutions $u_1,u_2$ whose corresponding vector field $X$ lets us recover the coefficients on a dense set in $\Om$. This follows from the argument given in the proof of Theorem 2.2 in \cite{BalRen2}, which actually shows that we can recover the coefficients almost everywhere; a version of this same argument is used to analyze the stability of the reconstruction in \cite{BonChoTri}.  This article presents an alternate proof for the density result by considering the flow on $\Om$ generated by $X$ and applying dynamical systems point of view. More precisely, we prove the following.  

\begin{theorem}\label{MainTheorem}
Suppose $f_1,f_2 \in C(\partial \Om)$ with $f_2$ positive and $f_1/f_2$ not constant.  Let $u_1$ and $u_2$ be the solutions to \eqref{Transport} with $u_1 = f_1$ and $u_2 = f_2$ on $\partial \Om$, and let $X$ be the vector field defined by \eqref{XDef}. Then the union of the integral curves of $X$ that intersect the boundary of $\Omega$ is dense in $\Omega$.  
\end{theorem} 

In other words, given continuous, positive, linearly independent boundary conditions, we can get arbitrarily close to any point in $\Omega$ from the boundary by following an integral curve of $X$. It follows that the transport method allows us to recover $D$ and $\sigma$ on a dense set without special care in selecting the boundary conditions $f_1$ and $f_2$. Note that if $D$ and $\sigma$ are a priori continuous, then we can recover $D$ and $\sigma$ on all of $\Omega$ by continuity. 

\section{Proof of Theorem \ref{MainTheorem}}

To begin, we will fix some notation. Let $u_1$, $u_2$, and $X$ be as in the statement of Theorem \ref{MainTheorem}, and make the following definitions.

\begin{defn}\label{Equivalence}
Let $x, y \in \bar{\Omega}$.  We say that $x \sim y$ if there exists an integral curve $\gamma:[0,b] \rightarrow \bar{\Omega}$ defined by $\dot{\gamma}(t) = X(\gamma(t))$ such that both $x$ and $y$ lie in the image of $\gamma$. 
\end{defn}

\begin{defn}\label{Sigma}
For a set $A \subset \bar{\Omega}$, define 
\[
\Sigma_A = \{ y \in \bar{\Omega} | y \sim x \mbox{ for some } x \in A\}.
\]
In other words, $\Sigma_A$ is the union of all integral curves of $X$ that intersect $A$.  
\end{defn}

With this notation, the statement of Theorem \ref{MainTheorem} is that the closure of $\Sigma_{\partial \Om}$ is the same as the closure of $\Om$; i.e.
\[
\overline{\Sigma_{\partial \Om}} = \bar{\Om}.
\]

Before beginning the proof of Theorem \ref{MainTheorem}, we make the following remark: since $D$ is uniformly positive, we can replace $X$ by $DX$ in Definition \ref{Equivalence}. In other words, the following definition is equivalent to Definition \ref{Equivalence}.  

\begin{defn}\label{EquivalentEquivalence}
Let $x, y \in \bar{\Omega}$.  We say that $x \sim y$ if there exists an integral curve $\gamma:[0,b] \rightarrow \bar{\Omega}$ defined by $\dot{\gamma}(t) = DX(\gamma(t))$ such that both $x$ and $y$ lie in the image of $\gamma$. 
\end{defn}

Indeed, if we have an integral curve $\gamma:[0,b] \rightarrow \bar{\Omega}$ defined by the equation $\dot{\gamma}(t) =  X(\gamma(t))$, we can define a function $g$ by the ODE
\[
\dot{g}(t) = D(\gamma(g(t))) \quad \mbox{ and } \quad g(0) = 0.
\]
Since $D$ is uniformly positive, $g$ is increasing, so there exists $b'$ such that $g(b') = b$.  Then we can define a new curve $\tilde{\gamma}:[0,b'] \rightarrow \bar{\Omega}$ by reparametrizing $\gamma$ with $g$:
\[
\tilde{\gamma}(t) = \gamma(g(t)).
\]
Now $\tilde{\gamma}([0,b']) = \gamma([0,b])$ and
\[
\dot{\tilde{\gamma}}(t) = DX(\tilde{\gamma}(t)).
\]
Therefore if $x \sim y$ according to Definition \ref{Equivalence} then $x \sim y$ according to Definition \ref{EquivalentEquivalence}, and the converse follows similarly.  With this in mind, we turn to the proof of Theorem \ref{MainTheorem}.

\begin{proof}[Proof of Theorem \ref{MainTheorem}]
Suppose that $\Omega \setminus \Sigma_{\partial \Omega}$ contains an open set $U$.  Then no integral curve of the vector field $DX$ joins any point of $U$ to $\partial \Om$, so it follows that $\Sigma_U$ is disjoint from $\partial \Omega$, and therefore $\Sigma_U \subset \Omega$.

Now the vector field $DX$ gives a flow on $\Sigma_U$, defined for all time, that maps $\Sigma_U$ to itself.  Moreover, 
\[
\grad \cdot D X = \grad \cdot D (u_2 \grad u_1 - u_1 \grad u_2)  = 0,
\]
so the vector field $DX$ is divergence free.  This means that the flow of $DX$ preserves volume, so the Poincar\'{e} Recurrence Theorem applies to maps defined by this flow.  This gives us the following result, (see e.g. \cite{Arn}, p71-72): 

\begin{prop}[Poincar\'e Recurrence Theorem] 
Let $W \subset \Sigma_U$ be open.  For $x \in W$ and $k \in \N$, define
\[
x_k = \gamma_x(k),
\]
where $\gamma_x$ is the integral curve defined by  $\dot{\gamma}_x(t) = DX(\gamma_x(t))$, with the initial condition $\gamma_x(0) = x$.  Then for almost every $x \in W$, $x_k \in W$ for infinitely many $k$.  
\end{prop}

The basic idea of the proof of Theorem \ref{MainTheorem} is as follows. A short calculation shows that 
\begin{equation}\label{XIdentity}
X = u_2^2 \grad u,
\end{equation}
where $u = u_1/u_2$. The maximum principle, together with the positivity of $f_2$, guarantees that $u_2$ is uniformly positive, so $u$ is well defined. Moreover the integral curves of $X$ and $DX$ are the same as the integral curves of $\grad u$, by the same logic used in the discussion of Definition \ref{EquivalentEquivalence}.  If any integral curve of $X$ were closed, we could integrate $\grad u$ along that curve and obtain two different values of $u$, which would be a contradiction.  The main idea of the proof is to apply the Poincar\'e Recurrence Theorem to a well chosen subset $W \subset \Sigma_U$, to provide us with a trajectory that approximates a closed curve well enough to force a contradiction.  

To obtain this subset $W$, define $u = u_1/u_2$.   Since $u$ is not constant at the boundary, unique continuation guarantees that $u$ is not constant on $\Sigma_U$.  Therefore there exists some point $y$ in $\Sigma_U$ such that 
\[
\left| \grad u(y) \right|  > 0.
\]
Then the regularity of $u_1$ and $u_2$ guarantees that there exists an open set $V \subset \Sigma_U$ containing $y$ and a positive constant $c$ such that $|\grad u| > c$ on $V$.  

Now consider an open set $W$ which contains $y$ and is compactly contained in $V$. Applying the Poincar\'e Recurrence Theorem to $W$, we see that there exists $x_0 \in W$ such that $x_k \in W$ for infinitely many $k$.  

Let $\{ x_{k_j}\}$ denote the subsequence of $\{x_k\}$ such that $x_{k_j} \in W$, and let $\gamma^j:[k_j,k_{j+1}] \rightarrow \Omega$ be the integral curve of $DX$ joining $x_{k_j}$ to $x_{k_{j+1}}$.  We can obtain $u(x_{k_{j+1}})$ from $u(x_{k_j})$ by integrating $\grad u$ over $\gamma^j$; in other words
\begin{equation}\label{LineIntegral}
u(x_{k_{j+1}}) - u(x_{k_j}) = \int_{\gamma_j} \grad u \cdot dr. 
\end{equation}
For each $j$, one of the following two things must happen:

\begin{itemize}

\item
Case I: the image of $\gamma^j$ is entirely contained in $V$, or 

\item
Case II: the image of $\gamma^j$ contains points outside $V$.  

\end{itemize}

In Case I, we can parametrize \eqref{LineIntegral} to get
\begin{eqnarray*}
u(x_{k_{j+1}}) - u(x_{k_j}) &=& \int_{k_j}^{k_{j+1}} \grad u(\gamma^j(t)) \cdot \dot{\gamma^j}(t) \, dt \\ 
			     &=& \int_{k_j}^{k_{j+1}} \grad u(\gamma^j(t)) \cdot DX(\gamma^j(t)) \, dt. \\ 
\end{eqnarray*}
Then \eqref{XIdentity} implies that 
\[
u(x_{k_{j+1}}) - u(x_{k_j})  = \int_{k_j}^{k_{j+1}} Du_2^2(\gamma^j(t))| \grad u(\gamma^j(t))|^2 \, dt. 
\]
Since the image of $\gamma^j$ is entirely contained in $V$, and $k_{j+1} - k_j \geq 1$, we have 
\[
u(x_{k_{j+1}}) - u(x_{k_j}) \geq \min_{\Omega} D u_2^2 \cdot c^2 > 0.
\]

In Case II, the length of the portion of $\gamma^j$ contained in $V$ must be at least twice the distance from $W$ to the exterior of $V$, so \eqref{LineIntegral} tells us that 
\[
u(x_{k_{j+1}}) - u(x_{k_j}) \geq 2c\, \mathrm{dist}(W,\mathrm{ext}\, V)>0.
\]

In both cases, $u(x_{k_{j+1}}) - u(x_{k_j})$ is bounded below uniformly in $j$. By setting $q$ to be the minimum of the bounds in both cases, 
we see that $u(x_{k_{j+1}}) - u(x_{k_j}) \geq q$ for each $j \in \N$, and therefore $u$ is unbounded in $W$.  But this contradicts the continuity of $u$, which is guaranteed by the continuity and positivity of $u_1$ and $u_2$, and so our initial supposition is false.  Therefore
$\overline{\Sigma_{\partial \Omega}} = \bar{\Om}$ as claimed.  

\end{proof}

As a final remark, note that if $\sigma \equiv 0$, we can take $u_2$ to be the identity function.  Then \eqref{XIdentity} implies that $X = \grad u_1$, and Theorem \ref{MainTheorem} gives us a neat corollary:

\begin{cor}
Suppose $u \in C^2(\Om) \cap C^1(\bar{\Om})$, and
\[
\grad \cdot D \grad u = 0
\]
in $\Omega$.  Then the set of integral curves of $\grad u$ that intersect the boundary of $\Omega$ is dense in $\Omega$.  
\end{cor}

\end{document}